\documentclass[oneside,english]{amsart}%
\usepackage{amsfonts}
\usepackage[T1]{fontenc}
\usepackage[latin9]{inputenc}
\usepackage{units}
\usepackage{amstext}
\usepackage{amssymb}
\usepackage{esint}
\usepackage{babel}
\usepackage{amsmath}
\usepackage{graphicx}%
\providecommand{\U}[1]{\protect\rule{.1in}{.1in}}
\newtheorem{theorem}{Theorem}
\theoremstyle{definition}
\newtheorem{remark}[theorem]{Remark}
\newtheorem{example}[theorem]{Example}
\theoremstyle{plain}
\newtheorem{corollary}[theorem]{Corollary}
\begin{document}
\title
[Well-posedness of Bayesian data assimilation and inverse problems]{On well-posedness of Bayesian data assimilation and inverse problems in Hilbert space}
\thanks
{Supported partially by the Czech Science Foundation grant 13-34856S and NSF grant 1216481}
\author{Ivan Kasanick\'y}
\address
{Institute of Computer Science, Academy of Sciences of the Czech Republic, Pod
Vodárenskou v\v{e}\v{z}\'{\i}\ 271/2, 182 07 Praha 8, Czech Republic}
\author{Jan Mandel}
\address
{University of Colorado Denver, Denver, CO 80217-3364, USA, and Institute of
Computer Science, Academy of Sciences of the Czech Republic, Pod Vodárenskou
v\v{e}\v{z}\'{\i} 271/2, 182 07 Praha 8, Czech Republic}
\begin{abstract}
Bayesian inverse problem on an infinite dimensional separable Hilbert space
with the whole state observed is well posed when the prior state distribution
is a Gaussian probability measure and the data error covariance is a cylindrical
Gaussian measure whose covariance has positive lower bound. If the state
distribution and the data distribution are equivalent Gaussian probability
measures, then the Bayesian posterior measure is not well defined. If the
state covariance and the data error covariance commute, then the Bayesian
posterior measure is well defined for all data vectors if and only if the data
error covariance has positive lower bound, and the set of data vectors for
which the Bayesian posterior measure is not well defined is dense if the data
error covariance does not have positive lower bound.
\end{abstract}
\maketitle

\section{Introduction}

Data assimilation and the solution of inverse problems on infinite dimensional
spaces are of interest as a limit case of discrete problems of an increasing
resolution and thus increasing dimension. Computer implementation is by
necessity finite dimensional, yet studying a discretized problem (such as a
finite difference or finite element discretization) as an approximation of an
infinite-dimensional one (such as a partial differential equation) is a basic
principle of numerical mathematics. This principle has recently found use in
Bayesian inverse problems as well, for much the same reason; important
insights in high-dimensional probability are obtained by considering it in the
light of infinite dimension. See \cite{Stuart-2010-IPB,Stuart-2013-BAI} for an
introduction and an extensive discussion.

Bayesian data assimilation and inverse problems are closely linked; the prior
distribution acts as a regularization and the maximum aposteriori probability
(MAP) delivers a single solution to the inverse problem. In the Gaussian case,
the prior becomes a type of Tikhonov regularization and the MAP\ estimate is
essentially a~regularized least squares solution. Since there is no Lebesque
measure in an infinite dimension, the standard probability density does not
exist, and the MAP\ estimate needs to be understood in a generalized sense
\cite{Cotter-2009-BIP,Dashti-2012-BPB}.

However, unlike in a finite dimension, even the simplest problems are often
ill-posed when the data is infinite dimensional. It is often assumed that the
data likelihood is such that the problem is well posed, or, more specifically,
that the data space is finite dimensional, e.g.,
\cite{Cotter-2009-BIP,Dashti-2012-BPB,Hosseini-2016-WPB-2,Hosseini-2016-WPB,Iglesias-2015-FBM,Stuart-2010-IPB,Stuart-2013-BAI}%
. Well-posedness of the infinite dimensional problem affects the performance
of stochastic filtering algorithms for finite dimensional approximations; it
was observed computationally \cite[Sec. 4.1]{Beezley-2009-HDA} that the
performance of the ensemble Kalman filter and the particle filter does not
deteriorate with increasing dimension when the state distribution approaches
a~Gaussian probability measure, but the curse of dimensionality sets in when
the state distribution approaches white noise. A related theoretical analysis
was recently developed in \cite{Agapiou-2015-ISC}.

It was noted in \cite{Lasanen-2007-MID} that Bayesian filtering is well
defined only for some values of observations when the data space is infinite
dimensional. In \cite{Agapiou-2015-ISC}, necessary and sufficient conditions
were given in the Gaussian case for the Bayesian inverse problem to be well
posed for all data vectors a.s. with respect to the data distribution, which
was understood as a Gaussian measure on a larger space than the given data
space. However, in the typical case studied here, such random data vector is
a.s. not in the given data space, so the conditions in \cite{Agapiou-2015-ISC}
are not informative in our setting. See Remark \ref{rem:agapiou} for more details.

In this paper, we study perhaps the simplest case of a Bayesian inverse
problem on an infinite dimensional separable Hilbert space: the whole state is
observed, the observation operator (the forward operator in inverse problems
nomenclature)\ is identity, and both the state distribution and the data error
distribution (which enters in the data likelihood) are Gaussian. The state
distribution is (a standard, $\sigma$-additive) Gaussian measure, but the data
error distribution is allowed to be only a weak Gaussian measure
\cite{Balakrishnan-1976-AFA}, that is, only a finitely additive cylindrical
measure \cite{Schwartz-1973-RMA}. This way, we may give up the $\sigma
$-additivity of the data error distribution, but the data vectors are in the
given data space. Weak Gaussian measure is $\sigma$-additive if and only if
its covariance is has finite trace. White noise, with covariance bounded away
from zero on infinite dimensional space, is an example of a weak Gaussian
measure, which is not $\sigma$-additive.

It is straightforward that when the data error covariance has positive lower
bound, then the least squares, Kalman filter, and the Bayesian posterior are
all well defined (Theorems \ref{thm:bounded-3dvar}, \ref{thm:boulded-kf}, and
\ref{thm:bounded-bayes}). The main results of this paper consist of the study
of the converse when the state is a Gaussian measure:

\begin{enumerate}
\item Example \ref{ex:3dvar}: If the state covariance and the data error
covariance are the same operator with finite trace, then the least squares are
not well posed for some data vectors.

\item Example \ref{ex:equivalent}: If the state distribution and the data
error distribution are equivalent Gaussian measures on infinite dimensional
space, then the posterior measure is not well defined.

\item Theorem \ref{thm:commute}: If the state covariance and the data error
covariance commute, then the posterior measure is well defined for all data
vectors if and only if the data error covariance has positive lower bound.

\item Corollary \ref{cor:dense}: If the state covariance and the data
covariance commute and the data covariance does not have positive lower bound,
then the set of vectors for which the posterior measure is not well defined is dense.
\end{enumerate}

The paper is organized as follows. In Section \ref{sec:notation}, we recall
some background and establish notation. The well-posedness of data
assimilation as a least squares problem is considered in Section
\ref{sec:3DVAR}, the well-posedness of Kalman filter formulas in Section
\ref{sec:KF}, and the well-posedness of the Bayesian setting in terms of
measures in Section \ref{sec:Bayes}.

\section{Notation}

\label{sec:notation}We denote by $\mathcal{H}$ a separable Hilbert space with
a real-valued inner product denoted by $\left\langle u,v\right\rangle $ and
the norm $\left\vert u\right\vert =\sqrt{\left\langle u,u\right\rangle }$. We
assume that $\mathcal{H}$ has infinite dimension, though all statements hold
in finite dimension as well. We denote by $\left[  \mathcal{H}\right]  $ the
space of all bounded linear operators from $\mathcal{H}$ to $\mathcal{H}.$ We
say that $\mathrm{R}\in\left[  \mathcal{H}\right]  $ has positive lower bound
if
\[
\left\langle \mathrm{R}u,u\right\rangle \geq\alpha\left\langle
u,u\right\rangle =\alpha\left\vert u\right\vert ^{2}%
\]
for some $\alpha>0$ and all $u\in\mathcal{H}.$ We write $\mathrm{R}>0$ when
$\mathrm{R}\in\left[  \mathcal{H}\right]  $ is symmetric, i.e., $\mathrm{R}%
=\mathrm{R}^{\ast}$, where $\mathrm{R}^{\ast}$ denotes the adjoint operator to
$\mathrm{R},$ and has positive lower bound. The operator $\mathrm{R}\in\left[
\mathcal{H}\right]  $ is positive semidefinite if
\[
\left\langle \mathrm{R}u,u\right\rangle \geq0
\]
for all $u\in\mathcal{H},$ and we use the notation $\mathrm{R}\geq0$ when
$\mathrm{R}$ is symmetric and positive semidefinite. We say that
$\mathrm{R\geq0}$ is a trace class operator if
\[
\operatorname*{Tr}\mathrm{R}=\sum_{i=1}^{\infty}\left\langle \mathrm{R}%
e_{i},e_{i}\right\rangle <\infty
\]
where $\left\{  e_{i}\right\}  $ is a total orthonormal set in $\mathcal{H}.$
$\operatorname*{Tr}\mathrm{R}$ does not depend on the choice of $\left\{
e_{i}\right\}  $.

We denote by $\mathcal{L}\left(  \mathcal{H}\right)  $ the space of all random
variables on $\mathcal{H}$, i.e., if $X\in\mathcal{L}\left(  \mathcal{H}%
\right)  $, then $X$ is a measurable mapping from a probability space $\left(
\Omega,\mathcal{A},P\right)  $ to $\left(  \mathcal{H},\mathcal{B}\left(
\mathcal{H}\right)  \right)  $ where $\mathcal{B}\left(  \mathcal{H}\right)  $
denotes Borel $\sigma$-algebra on $\mathcal{H}$. A weak random variable
$W\in\mathcal{L}_{w}\left(  \mathcal{H}\right)  $ is a mapping
\[
W:\ \left(  \Omega,\mathcal{A},P\right)  \rightarrow\left(  \mathcal{H}%
,\mathcal{C}\left(  \mathcal{H}\right)  \right)  ,
\]
where $\left(  \Omega,\mathcal{A},P\right)  $ is a general probability space,
and $\mathcal{C}\left(  \mathcal{H}\right)  $ denotes an algebra of
cylindrical sets on $\mathcal{H}$, such that:

\begin{enumerate}
\item for all $D\in\mathcal{C}\left(  \mathcal{H}\right)  $, it holds that
$W^{-1}\left(  D\right)  \in\mathcal{A}$, and

\item for any $n\in\mathbb{N}$ and any $e_{1},\ldots,e_{n}\in\mathcal{H}$ the
mapping
\[
V:\ \omega\in\Omega\rightarrow\left(  \left\langle e_{1},W\left(
\omega\right)  \right\rangle ,\ldots,\left\langle e_{n},W\left(
\omega\right)  \right\rangle \right)  ^{\ast}\in\mathbb{R}^{n}%
\]
is measurable, i.e., $V$ is an $n$-dimensional real random vector.
\end{enumerate}

We denote by $\mathcal{L}_{w}\left(  \mathcal{H}\right)  $ the space of all
weak random variables on $\mathcal{H}$. Obviously, when $\mathrm{dim}\left(
\mathcal{H}\right)  <\infty,$ then $\mathcal{L}_{w}\left(  \mathcal{H}\right)
=\mathcal{L}\left(  \mathcal{H}\right)  ,$ i.e., weak random variables are
interesting only if the dimension of the state is infinite. A weak random
variable $\mbox{W\ensuremath{\in}}\mathcal{L}_{w}\left(  \mathcal{H}\right)  $
has weak Gaussian distribution (also called a cylindrical measure), denoted
$W\sim\mathcal{N}\left(  m,\mathrm{R}\right)  ,$ where $m\in\mathcal{H}$ is
the mean of W, and $\mathrm{R}\in\left[  \mathcal{H}\right]  ,\mathrm{R}%
\geq0,$ is the covariance of $W$ if, for any finite set $\left\{  e_{1}%
,\ldots,e_{n}\right\}  \subset\mathcal{H},$ the random vector $\left(
\left\langle e_{1},W\right\rangle ,\ldots,\left\langle e_{n},W\right\rangle
\right)  ^{\ast}$ has multivariate Gaussian distribution $\mathcal{N}\left(
\mu,\Sigma\right)  $ with
\[
\mu=\left(  \mathrm{R}e_{1},\mathrm{\ldots,R}e_{n}\right)  ^{\ast}%
\in\mathbb{R}^{n}%
\]
and covariance matrix $\Sigma\in\mathbb{R}^{n\times n},$
\[
\left[  \Sigma\right]  _{i,j}=\left\langle e_{i},\mathrm{R}e_{j}\right\rangle
,
\]
where $\left[  \Sigma\right]  _{i,j}$ denotes the element of the matrix
$\Sigma$ in the $i^{\text{th}}$ row and the $j^{\text{th}}$ column. It can be
shown that $W\sim\mathcal{N}\left(  m,\mathrm{R}\right)  $ is measurable,
i.e., $W\in\mathcal{L}\left(  \mathcal{H}\right)  ,$ if and only if the
covariance $\mathrm{R}$ is trace class, e.g., \cite[Theorem 6.2.2]%
{Balakrishnan-1976-AFA}.

For further background on probability on infinite dimensional spaces and
cylindrical measures see, e.g.,
\cite{Bogachev-1998-GM,DaPrato-1992-SEI,Vakhania-1987-PDB}.

\section{Data assimilation}

Suppose that $\left\{  X^{\left(  t\right)  }\right\}  _{t\in\mathbb{N}}$ is a
dynamical system defined on a separable Hilbert space $\mathcal{H}.$ Data
assimilation uses observations of the form
\[
Y^{\left(  t\right)  }=\mathrm{H}X^{\left(  t\right)  }+W^{\left(  t\right)
},\ t\in\mathbb{N},
\]
where $\mathrm{H}\in\left[  \mathcal{H}\right]  $, and $W^{\left(  t\right)
}\sim\mathcal{N}\left(  0,\mathrm{R}\right)  $, to estimate sequentially the
states of the dynamical system. In each data assimilation cycle, i.e., for
each $t\in\mathbb{N},$ a forecast state $X^{\left(  t\right)  ,f}$ is combined
with the observation $Y^{\left(  t\right)  }$ to produce a better estimate of
the true state $X^{\left(  t\right)  }.$ Hence, one data assimilation cycle is
an inverse problem \cite{Stuart-2010-IPB,Stuart-2013-BAI,Cotter-2009-BIP}.
Since we are interested in one data assimilation cycle only, we drop the time
index for the rest of the paper.

\subsection{3DVAR}

\label{sec:3DVAR}The 3DVAR method is based on a minimization of the cost
function%
\begin{equation}
J^{\text{3DVAR}}\left(  x\right)  =\left|  x-x^{f}\right|  _{\mathrm{B}^{-1}%
}^{2}+\left|  y-x\right|  _{\mathrm{R}^{-1}}^{2} \label{eq:3DVAR}%
\end{equation}
where $\mathrm{B}$ is a known background covariance operator and $\mathrm{R}$
is a data noise covariance. If the state space is finite dimensional, and the
matrix $\mathrm{B}$ is regular, then the norm on the right-hand side of
(\ref{eq:3DVAR}) is defined by
\[
\left|  x\right|  _{\mathrm{B}^{-1}}^{2}=\left\langle \mathrm{B}%
^{-\nicefrac{1}{2}}x,\mathrm{B}^{-\nicefrac{1}{2}}x\right\rangle
,\ x\in\mathcal{H}.
\]
However, when the state space is infinite dimensional, the inverse of a
compact linear operator is unbounded and only densely defined. It is then
natural to extend the quadratic forms on the right-hand side of
(\ref{eq:3DVAR}) as%
\begin{equation}
\left\vert x\right\vert _{\mathrm{B}^{-1}}^{2}=%
\begin{cases}
\left\langle \mathrm{B}^{-\nicefrac{1}{2}}x,\mathrm{B}^{-\nicefrac{1}{2}}%
x\right\rangle  & \text{if}\ x\in\mathrm{B}^{\nicefrac{1}{2}}\left(
\mathcal{H}\right)  ,\\
\infty & \text{if}\ x\notin\mathrm{B}^{\nicefrac{1}{2}}\left(  \mathcal{H}%
\right)  ,
\end{cases}
\label{eq:3dvar:B-norm}%
\end{equation}
where $\mathrm{B}^{\nicefrac{1}{2}}\left(  \mathcal{H}\right)  =\mathrm{Im}%
\left(  \mathrm{B}^{\nicefrac{1}{2}}\right)  ,$ i.e., $\mathrm{B}%
^{\nicefrac{1}{2}}\left(  \mathcal{H}\right)  $ denotes the image of the
operator $\mathrm{B}^{\nicefrac{1}{2}}$, and
\begin{equation}
\left\vert x\right\vert _{\mathrm{R}^{-1}}^{2}=%
\begin{cases}
\left\langle \mathrm{R}^{-\nicefrac{1}{2}}x,\mathrm{R}^{-\nicefrac{1}{2}}%
x\right\rangle  & \text{if}\ x\in\mathrm{R}^{\nicefrac{1}{2}}\left(
\mathcal{H}\right)  ,\\
\infty & \text{if}\ x\notin\mathrm{\mathrm{R}}^{\nicefrac{1}{2}}\left(
\mathcal{H}\right)  .
\end{cases}
\label{eq:3dvar:R-norm}%
\end{equation}
Obviously, the 3DVAR cost function attains infinite value, and, even worse, it
is not hard to construct an example when $J^{\text{3DVAR}}\left(  x\right)
=\infty$ for all $x\in\mathcal{H}$.

\begin{example}
\label{ex:3dvar}Suppose that $\mathrm{\mathrm{R}}=\mathrm{B},$ $\mathrm{B}$ is
a trace class operator,
\begin{equation}
x^{f}\in\mathrm{B}^{\nicefrac{1}{2}}\left(  \mathcal{H}\right)  ,
\label{eq:3dvar:x:in:B}%
\end{equation}
and
\begin{equation}
y\notin\mathrm{B}^{\nicefrac{1}{2}}\left(  \mathcal{H}\right)  .
\label{eq:3dvar:y:in:B}%
\end{equation}
If $x\notin\mathrm{B}^{\nicefrac{1}{2}}\left(  \mathcal{H}\right)  ,$ then
\[
\left(  x-x^{f}\right)  \notin\mathrm{B}^{\nicefrac{1}{2}}\left(
\mathcal{H}\right)
\]
because $\mathrm{B}^{\nicefrac{1}{2}}\left(  \mathcal{H}\right)  $ is an
linear subspace of $\mathcal{H}$ and (\ref{eq:3dvar:x:in:B}), so
\[
J^{\text{3DVAR}}\left(  x\right)  =\infty.
\]
When $x\in\mathrm{B}^{\nicefrac{1}{2}}\left(  \mathcal{H}\right)  $, then
\[
\left(  y-x\right)  \notin\mathrm{B}^{\nicefrac{1}{2}}\left(  \mathcal{H}%
\right)
\]
using (\ref{eq:3dvar:y:in:B}), so, again,
\[
J^{\text{3DVAR}}\left(  x\right)  =\infty.
\]
Therefore, $J^{\text{3DVAR}}\left(  x\right)  =\infty$ for all $x\in
\mathcal{H}.$
\end{example}

Naturally, a minimization of $J^{\text{3DVAR}}\left(  x\right)  $ does not
make sense unless there is at least one $x\in\mathcal{H}$ such that
$J^{\text{3DVAR}}\left(  x\right)  <\infty.$ Fortunately, we can formulate
a~sufficient condition when this condition is fulfilled.

\begin{theorem}
\label{thm:bounded-3dvar}If at least one of the operators $\mathrm{B}$ and
$\mathrm{R}$ has positive lower bound, then for any possible values of $x^{f}$
and $y,$ there exist at least one $x\in\mathcal{H}$ such that
\[
J^{\text{3DVAR}}\left(  x\right)  <\infty.
\]

\end{theorem}

\begin{proof}
Without loss of generality, assume that $\mathrm{R}$ has positive lower bound.
Hence,
\[
\left|  x-y\right|  _{\mathrm{R}^{-1}}^{2}<\infty
\]
for any combinations of $x\in\mathcal{H}$ and $y\in\mathcal{H}.$ Therefore,
given $x^{f}\in\mathcal{H},$
\[
J^{\text{3DVAR}}\left(  x\right)  <\infty
\]
for any $x\in\left\{  z\in\mathrm{B}^{\nicefrac{1}{2}}%
\mbox{\ensuremath{\left(\mathcal{H}\right)}}:z=x-x^{f}\right\}  .$
\end{proof}

\subsection{KF and EnKF}

\label{sec:KF}The ensemble Kalman filter
\cite{Evensen-1994-SDA,Houtekamer-1998-DAE}, which is based on the Kalman
filter \cite{Kalman-1960-NAL,Kalman-1961-NRF}, is one of the most popular
assimilation method. The key part of both methods is the Kalman gain operator
\begin{equation}
\mathcal{K}:\ \mathrm{P}\mathbb{\mapsto}\mathrm{H}\mathrm{P}^{f}%
\mathrm{H}^{\ast}\left(  \mathrm{H}\mathrm{P}^{f}\mathrm{H}^{\ast}%
+\mathrm{R}\right)  ^{-1}. \label{eq:KGO}%
\end{equation}
where $\mathrm{P\in\left[  \mathcal{H}\right]  }$ and $\mathrm{P}>0$. If the
data space is finite dimensional, then the matrix $\mathrm{H}\mathrm{P}%
\mathrm{H}^{\ast}+\mathrm{R}$ is positive definite, and the inverse is well
defined. However, when data space is infinite dimensional, the operator
$\mathrm{H}\mathrm{P}\mathrm{H}^{\ast}+\mathrm{R}$ may not be defined on the
whole space since an inverse of a trace class operator is only densely
defined. Therefore, the KF update equation
\[
X^{a}=X^{f}+\mathcal{K}\left(  \mathrm{P}^{f}\right)  \left(  y-\mathrm{H}%
X^{f}\right)  ,
\]
where $\mathrm{P}^{f}=\mathrm{cov}\left(  X^{f}\right)  $, may not be
applicable since there is no guarantee that the term $\mathcal{K}\left(
\mathrm{P}^{f}\right)  \left(  y-\mathrm{H}X^{f}\right)  $ is defined. Yet,
similarly to 3DVAR, there is a sufficient condition when the Kalman filter
algorithm is well defined for any possible values.

\begin{theorem}
\label{thm:boulded-kf}If the data noise covariance $\mathrm{R}$ has positive
lower bound, then the Kalman gain operator $\mathcal{K}\left(  \mathrm{P}%
^{f}\right)  $ is defined on the whole space $\mathcal{H}$.
\end{theorem}

\begin{proof}
If $\mathrm{R}$ has positive lower bound, then the linear operator
$\mathrm{H}\mathrm{P}^{f}\mathrm{H}^{\ast}+\mathrm{R}$ has positive lower
bound as well because $\mathrm{P}$ is the covariance operator, so
$\mathrm{P}>0.$ The statement now follows from the fact that an operator with
positive lower bound has an inverse defined on the whole space.
\end{proof}

\section{Bayesian approach}

\label{sec:Bayes}Denote by $\mu^{f}$ the distribution of $X^{f}.$ Bayes'
theorem prescribes the analysis measure by
\begin{equation}
\mu^{a}\left(  B\right)  \propto\int_{B}d\left(  y\left\vert x\right.
\right)  d\mu^{f}\left(  x\right)  \label{eq:Bayes-update}%
\end{equation}
for all $B\in\mathcal{B}\left(  \mathcal{H}\right)  $ if
\begin{equation}
c\left(  y\right)  =\int_{\mathcal{H}}d\left(  y\left\vert x\right.  \right)
d\mu^{f}\left(  x\right)  >0, \label{eq:bayes-constant}%
\end{equation}
where the given function $d:\ \mathcal{H}\times\mathcal{H}\rightarrow\left[
0,\infty\right)  $ is called a data likelihood. If the distribution of the
forecast and data noise are both Gaussian, then
\begin{equation}
d\left(  y\left\vert x\right.  \right)  \propto\exp\left(  -\frac{1}%
{2}\left\vert y-x\right\vert _{\mathrm{R}^{-1}}^{2}\right)  ,
\label{eq:likelihood}%
\end{equation}
where
\[
\left\vert x\right\vert _{\mathrm{R}^{-1}}^{2}=%
\begin{cases}
\left\langle \mathrm{R}^{-\nicefrac{1}{2}}x,\mathrm{R}^{-\nicefrac{1}{2}}%
x\right\rangle  & \text{if}\ x\in\mathrm{R}^{\nicefrac{1}{2}}\left(
\mathcal{H}\right)  ,\\
\infty & \text{if}\ x\notin\mathrm{\mathrm{R}}^{\nicefrac{1}{2}}\left(
\mathcal{H}\right)  .
\end{cases}
\]
With the natural convention that $\exp\left(  -\infty\right)  =0$, we have%
\[
d\left(  y\left\vert x\right.  \right)  =0\text{ if }x\notin\mathrm{\mathrm{R}%
}^{\nicefrac{1}{2}}\left(  \mathcal{H}\right)  .
\]

When both state and data spaces are finite dimensional, condition
(\ref{eq:bayes-constant}) is fulfilled for any possible value of observation
$y\in\mathcal{H}.$ Unfortunately, when both spaces are infinite dimensional,
condition (\ref{eq:bayes-constant}) may not be fulfilled as shown in the next example.

\begin{example}
\label{ex:equivalent}Assume that $X^{f}\sim\mathcal{N}\left(  m^{f}%
,\mathrm{P}^{f}\right)  $ and, $m^{f}$ belongs to the Cameron-Martin space of
$X^{f},$ i.e., $m^{f}\in\left(  \mathrm{P}^{f}\right)  ^{\nicefrac{1}{2}}%
\left(  \mathcal{H}\right)  $. If the measures $\mu^{f}$ and $\mu_{\mathrm{R}%
}\sim\mathcal{N}\left(  0,\mathrm{R}\right)  $ are equivalent, then both have
the same Cameron-Martin space, and%
\[
\mu^{f}\left(  \mathrm{R}^{\nicefrac{1}{2}}\left(  \mathcal{H}\right)
\right)  =\mu_{\mathrm{R}}\left(  \mathrm{R}^{\nicefrac{1}{2}}\left(
\mathcal{H}\right)  \right)  =0,
\]
so
\begin{align*}
\int_{\mathcal{H}}d\left(  0\left\vert x\right.  \right)  d\mu^{f}\left(
x\right)  =  &  \int_{\mathrm{R}^{\nicefrac{1}{2}}\left(  \mathcal{H}\right)
}\exp\left(  -\frac{1}{2}\left\vert x\right\vert _{\left(  \mathrm{R}\right)
^{-1}}^{2}\right)  d\mu^{f}\left(  x\right) \\
&  +\int_{\mathcal{H}\setminus\mathrm{R}^{\nicefrac{1}{2}}\left(
\mathcal{H}\right)  }\exp\left(  -\infty\right)  d\mu^{f}\left(  x\right)  =0.
\end{align*}

\end{example}

\begin{remark}
Another data likelihood is proposed in \cite{Stuart-2010-IPB},%
\[
\widetilde{d}\left(  y\left\vert x\right.  \right)  =%
\begin{cases}
d\left(  y\left\vert x\right.  \right)  & \text{if }\ c\left(  y\right)  >0,\\
1 & \text{if }\ c\left(  y\right)  =0,
\end{cases}
\]
where $c\left(  y\right)  $ is defined by (\ref{eq:bayes-constant}), and
$d\left(  y\left\vert x\right.  \right)  $ is defined by (\ref{eq:likelihood}%
). This definition leads to the analysis distribution
\begin{equation}
\widetilde{\mu}^{a}\left(  B\right)  =%
\begin{cases}
\frac{1}{c\left(  y\right)  }\int_{B}d\left(  y\left\vert x\right.  \right)
d\mu^{f}\left(  x\right)  & \text{if}\ c\left(  y\right)  >0,\\
\mu^{f}\left(  B\right)  & \text{if}\ c\left(  y\right)  =0
\end{cases}
\label{eq:mu_a:Stuart}%
\end{equation}
for all $B\in\mathcal{B}\left(  \mathcal{H}\right)  .$ That is, any data $y$
such that $c\left(  y\right)  =0$ is ignored.
\end{remark}

Obviously, the Bayesian update (\ref{eq:Bayes-update}), is useful only if the
set
\[
A=\left\{  y\in\mathcal{H}:\quad\int_{\mathcal{H}}d\left(  y\left|  x\right.
\right)  d\mu^{f}\left(  x\right)  =0\right\}
\]
is empty. The sufficient condition when the set $A$ is empty is similar to
conditions when previously mentioned assimilation techniques are well defined.

\begin{theorem}
\label{thm:bounded-bayes}The set
\[
A=\left\{  y\in\mathcal{H}:\quad\int_{\mathcal{H}}d\left(  y\left\vert
x\right.  \right)  d\mu^{f}\left(  x\right)  =0\right\}  ,
\]
where $\mu^{f}\sim\mathcal{N}\left(  m^{f},\mathrm{P}^{f}\right)  ,$ and the
data likelihood is defined by (\ref{eq:likelihood}), is empty if the operator
$\mathrm{R}$ has positive lower bound.
\end{theorem}

\begin{proof}
The operator $\mathrm{R}$ has positive lower bound, so the data likelihood
function
\[
d\left(  y\left\vert x\right.  \right)  \propto\exp\left(  -\frac{1}%
{2}\left\vert y-x\right\vert _{\mathrm{R}^{-1}}^{2}\right)
\]
is positive for any $x,y\in\mathcal{H}$, and it follows that
\[
\int_{\mathcal{H}}d\left(  y\left\vert x\right.  \right)  d\mu^{f}\left(
x\right)  >0
\]
for all $y\in\mathcal{H}$.
\end{proof}

In the special case when both forecast and data covariances commute, we can
show that this condition is also necessary for the set $A$ to be empty. Recall
that operators $\mathrm{P}^{f}$ and $\mathrm{R}$ commute when
\[
\mathrm{P}^{f}\mathrm{R}-\mathrm{RP}^{f}=0.
\]

\begin{theorem}
\label{thm:commute}Assume that $\mu^{f}\sim\mathcal{N}\left(  m^{f}%
,\mathrm{P}^{f}\right)  $, and operators $\mathrm{P}^{f}$ and $\mathrm{R}$
commute. Then
\[
\int_{\mathcal{H}}\exp\left(  -\frac{1}{2}\left\vert y-x\right\vert
_{\mathrm{R}^{-1}}^{2}\right)  d\mu^{f}\left(  x\right)  >0
\]
for all $y\in\mathcal{H}$ if and only if the operator $\mathrm{R}$ has
positive lower bound.
\end{theorem}

\begin{proof}
Without loss of generality assume that $m^{f}=0.$ The operators $\mathrm{P}%
^{f}$ and $\mathrm{R}$ are symmetric, commute, and $\mathrm{P}^{f}$ is
compact, so there exists a total orthonormal set $\left\{  e_{i}\right\}  $ of
common eigenvectors,%
\[
\mathrm{P}^{f}e_{i}=p_{i}e_{i},\ \mathrm{R}e_{i}=r_{i}e_{i},\text{ for all
}i\in\mathbb{N,}%
\]
e.g., \cite[Lemma 8]{Kasanicky-2016-EKF}, \cite{Levin-2002-IQT}, \cite[Section
II.10]{Neumann-1955-MFQ}.

For any $z\in\mathcal{H}$, denote by $\left\{  z_{i}\right\}  $ its Fourier
coefficient with respect to the orthonormal set $\left\{  e_{i}\right\}  $,
\[
z_{i}=\left\langle z,e_{i}\right\rangle ,\ i\in\mathbb{N}.
\]
Using this notation,
\[
d\left(  y|x\right)  =\exp\left(  -\frac{1}{2}\left\vert y-x\right\vert
_{\mathrm{R}^{-1}}^{2}\right)  =\prod_{i=1}^{\infty}\exp\left(  -\frac{\left(
y_{i}-x_{i}\right)  ^{2}}{2r_{i}}\right)  ,
\]
and
\[
\int_{\mathcal{H}}d\left(  y|x\right)  d\mu^{f}\left(  x\right)
=\int_{\mathcal{H}}\prod_{i=1}^{\infty}\exp\left(  -\frac{\left(  y_{i}%
-x_{i}\right)  ^{2}}{2r_{i}}\right)  d\mu^{f}\left(  x\right)  .
\]
Denote
\[
f_{n}\left(  x\right)  =\prod_{i=1}^{n}\exp\left(  -\frac{\left(  y_{i}%
-x_{i}\right)  ^{2}}{2r_{i}}\right)  ,\ n\in\mathbb{N}.
\]
Since $0<e^{-s^{2}}\leq1$ for any $s\in\mathbb{R}$, $\left\{  f_{n}\right\}  $
is a monotone sequence of functions on $\mathcal{H}$. The functions $f_{n}$
are continuous and therefore measurable, and by the the monotone convergence
theorem,
\begin{equation}
\int_{\mathcal{H}}d\left(  y|x\right)  d\mu^{f}\left(  x\right)
=\lim_{n\rightarrow\infty}\left(  \int_{\mathcal{H}}\prod_{i=1}^{n}\exp\left(
-\frac{\left(  y_{i}-x_{i}\right)  ^{2}}{2r_{i}}\right)  d\mu^{f}\left(
x\right)  \right)  . \label{eq:int:Bayes:2}%
\end{equation}

For each $i\in\mathbb{N}$, the random variable $\left\langle X^{f}%
,e_{i}\right\rangle $ has $\mathcal{N}\left(  0,p_{i}^{f}\right)  $
distribution, which we denote by $\mu_{i}^{f}$. Additionally,
\[
\mathrm{E}\left(  \left\langle X^{f},e_{i}\right\rangle \left\langle
X^{f},e_{j}\right\rangle \right)  =\delta_{ij},\ i,j\in\mathbb{N},
\]
and, in particular, the random variables $\left\langle X^{f},e_{i}%
\right\rangle $ and $\left\langle X^{f},e_{j}\right\rangle $ are independent
unless $i=j.$ Then,
\[
\int_{\mathcal{H}}f_{n}\left(  x\right)  d\mu^{f}\left(  x\right)
=\int_{\mathcal{H}}f_{n}\left(  x_{1},\ldots,x_{n},\ldots\right)  d\mu_{1}%
^{f}\left(  x_{1}\right)  \times\cdots\times d\mu_{n}^{f}\left(  x_{n}\right)
\]
for all $n\in\mathbb{N},$ and, using Fubini's theorem,
\[
\int_{\mathcal{H}}f_{n}\left(  x\right)  d\mu^{f}\left(  x\right)
=\int_{\mathbb{R}}\cdots\int_{\mathbb{R}}f_{n} \left(  x_{1},\ldots,x_{n}
\right)  d\mu_{1}^{f}\left(  x_{1}\right)  \cdots d\mu_{n}^{f}\left(
x_{n}\right)  .
\]
Now (\ref{eq:int:Bayes:2}) yields that
\[
\int_{\mathcal{H}}d\left(  y|x\right)  d\mu^{f}\left(  x\right)
=\lim_{n\rightarrow\infty}\prod_{i=1}^{n}\int_{\mathbb{R}}\exp\left(
-\frac{\left(  y_{i}-x_{i}\right)  ^{2}}{2r_{i}}\right)  d\mu_{i}^{f}\left(
x_{i}\right)  ,
\]
and, since the measure $\mu_{i}^{f}$ is absolutely continuous with respect to
the Lebesgue measure $\lambda$ on $\mathbb{R}$,
\begin{equation}
\int_{\mathcal{H}}d\left(  y|x\right)  d\mu^{f}\left(  x\right)  =\prod
_{i=1}^{\infty}\int_{-\infty}^{\infty}\exp\left(  -\frac{\left(  y_{i}%
-x_{i}\right)  ^{2}}{2r_{i}}\right)  \psi\left(  x_{i}\right)  d\lambda\left(
x_{i}\right)  \label{eq:prod:int:exp}%
\end{equation}
where
\[
\psi_{i}\left(  x\right)  =\frac{1}{\sqrt{2\pi p_{i}^{f}}}\exp\left(
-\frac{x_{i}^{2}}{2p_{i}^{f}}\right)  ,
\]
i.e., $\psi_{i}$ is the density of a $\mathcal{N}\left(  0,p_{i}^{f}\right)
$-distributed random variable.

The identity
\begin{align*}
\frac{\left(  y_{i}-x_{i}\right)  ^{2}}{r_{i}}+\frac{x_{i}^{2}}{p_{i}^{f}}  &
=\left(  \frac{1}{p_{i}^{f}}+\frac{1}{r_{i}}\right)  x_{i}^{2}-2\frac
{x_{i}y_{i}}{r_{i}}+\frac{y_{i}^{2}}{r_{i}}\\
&  =\frac{\left(  x_{i}-m_{i}^{a}\right)  ^{2}}{p_{i}^{a}}+\frac{y_{i}^{2}%
}{r_{i}+p_{i}^{f}},
\end{align*}
with
\[
m_{i}^{a}=\frac{p_{i}^{f}}{r_{i}+p_{i}^{f}}y_{i}\quad\text{and}\quad p_{i}%
^{a}=\left(  \frac{1}{p_{i}^{f}}+\frac{1}{r_{i}}\right)  ^{-1},
\]
allows us to write (\ref{eq:prod:int:exp}) in the form
\[
\int_{\mathcal{H}}d\left(  y|x\right)  d\mu^{f}\left(  x\right)  =\prod
_{i=1}^{\infty}\frac{1}{\sqrt{2\pi p_{i}^{f}}}\int_{-\infty}^{\infty}%
\exp\left(  -\frac{\left(  x_{i}-m_{i}^{a}\right)  ^{2}}{2p_{i}^{a}}%
-\frac{y_{i}^{2}}{2\left(  r_{i}+p_{i}^{f}\right)  }\right)  d\lambda\left(
x_{i}\right)  .
\]
By standard properties of the normal distribution,
\[
\int_{-\infty}^{\infty}\exp\left(  -\frac{\left(  x_{i}-m_{i}^{a}\right)
^{2}}{2p_{i}^{a}}\right)  dx_{i}=\sqrt{2\pi p_{i}^{a}}%
\]
for each $i\in\mathbb{N}$, so
\begin{align}
\int_{\mathcal{H}}d\left(  y|x\right)  d\mu^{f}\left(  x\right)  =  &
\prod_{i=1}^{\infty}\left(  \frac{p_{i}^{a}}{p_{i}^{f}}\right)
^{\nicefrac{1}{2}}\exp\left(  -\frac{y_{i}^{2}}{2\left(  r_{i}+p_{i}%
^{f}\right)  }\right) \nonumber\\
=  &  \prod_{i=1}^{\infty}\left(  1+\frac{p_{i}^{f}}{r_{i}}\right)
^{-\nicefrac{1}{2}}\exp\left(  -\frac{y_{i}^{2}}{2\left(  r_{i}+p_{i}%
^{f}\right)  }\right)  , \label{eq:inf-prod}%
\end{align}
where we used the computation
\[
\frac{p_{i}^{a}}{p_{i}^{f}}=\frac{1}{p_{i}^{f}\left(  \frac{1}{p_{i}^{f}%
}+\frac{1}{r_{i}}\right)  }=\frac{1}{1+\frac{p_{i}^{f}}{r_{i}}}.
\]
The infinite product (\ref{eq:inf-prod}) is nonzero if and only if the
following sum converges,
\begin{align}
&  \sum_{i=1}^{\infty}\log\left(  \left(  1+\frac{p_{i}^{f}}{r_{i}}\right)
^{-\nicefrac{1}{2}}\exp\left(  -\frac{y_{i}^{2}}{2\left(  r_{i}+p_{i}%
^{f}\right)  }\right)  \right) \nonumber\\
&  =-\frac{1}{2}\left(  \sum_{i=1}^{\infty}\log\left(  1+\frac{p_{i}^{f}%
}{r_{i}}\right)  \right)  -\left(  \sum_{i=1}^{\infty}\frac{y_{i}^{2}}%
{r_{i}+p_{i}^{f}}\right)  . \label{eq:sum:eig:val}%
\end{align}

To conclude the proof, we need to show that that (\ref{eq:sum:eig:val})
converges if and only if
\begin{equation}
r=\inf_{i\in\mathbb{N}}\left\{  r_{i}\right\}  >0. \label{eq:cond:inf:r}%
\end{equation}

First, the equivalence
\begin{equation}
\sum_{i=1}^{\infty}\ln\left(  1+\frac{p_{i}^{f}}{r_{i}}\right)  <\infty
\quad\Leftrightarrow\quad\sum_{i=1}^{\infty}\frac{p_{i}^{f}}{r_{i}}%
<\infty\label{eq:sum:ln:convergence}%
\end{equation}
follows from the limit comparison test because
\[
\lim_{i\rightarrow\infty}\frac{\ln\left(  1+\frac{p_{i}^{f}}{r_{i}}\right)
}{\frac{p_{i}^{f}}{r_{i}}}=1
\]
when
\begin{equation}
\lim_{i\rightarrow\infty}\frac{p_{i}^{f}}{r_{i}}=0. \label{eq:lim:p/r}%
\end{equation}
If condition (\ref{eq:lim:p/r}) is not satisfied, then both sums in
(\ref{eq:sum:ln:convergence}) diverge. If $r>0$, then the sum
\[
\sum_{i=1}^{\infty}\frac{p_{i}^{f}}{r_{i}}\leq\sum_{i=1}^{\infty}\frac
{p_{i}^{f}}{r}\leq r^{-1}\sum_{i=1}^{\infty}p_{i}^{f},
\]
and this sum converges because $\mathrm{P}^{f}$ is trace class.

Further, if $r>0$, then
\[
\sum_{i=1}^{\infty}\frac{y_{i}^{2}}{r_{i}+p_{i}^{f}}\leq\sum_{i=1}^{\infty
}\frac{y_{i}^{2}}{r}\leq r^{-1}\sum_{i=1}^{\infty}y_{i}^{2}=r^{-1}\left\vert
y\right\vert ^{2}<\infty
\]
since $\left\{  y_{i}\right\}  $ are Fourier coefficients of $y.$ On the other
side, when $r=0$, we will construct $\widetilde{y}\in\mathcal{H}$ such that
$\left\vert \widetilde{y}\right\vert \leq1$ and
\[
\sum_{i=1}^{\infty}\frac{\widetilde{y}_{i}^{2}}{r_{i}+p_{i}^{f}}=\infty.
\]
Since $r=0,$ there exists a subsequence $\left\{  r_{i_{k}}\right\}
_{k=1}^{\infty}$ such that
\[
r_{i_{k}}\leq\frac{1}{2^{k}},\quad k\in\mathbb{N},
\]
and we define
\[
\widetilde{y}=\sum_{i=1}^{\infty}\widetilde{y}_{i}e_{i}%
\]
with
\[
\widetilde{y}_{i}=%
\begin{cases}
r_{i}^{\nicefrac{1}{2}} & \text{if }i\in\left\{  i_{k}\right\}  _{k\in
\mathbb{N}},\\
0 & \text{if }i\notin\left\{  i_{k}\right\}  _{k\in\mathbb{N}.}%
\end{cases}
\]
The element $\widetilde{y}$ lies in the unit circle because
\[
\left\vert \widetilde{y}\right\vert ^{2}=\sum_{i=1}^{\infty}\widetilde{y}%
_{i}^{2}=\sum_{k=1}^{\infty}r_{i_{k}}\leq\sum_{k=1}^{\infty}\frac{1}{2^{k}%
}=1,
\]
while
\[
\sum_{i=1}^{\infty}\frac{\widetilde{y}_{i}^{2}}{r_{i}+p_{i}^{f}}=\sum
_{k=1}^{\infty}\frac{r_{i_{k}}}{r_{i_{k}}+p_{i_{k}}^{f}}=\sum_{k=1}^{\infty
}\frac{1}{1+\frac{p_{i_{k}}^{f}}{r_{i_{k}}}}=\infty
\]
where the last equality follows immediately from (\ref{eq:lim:p/r}).

Therefore, the sum (\ref{eq:sum:eig:val}) is finite for all $y\in\mathcal{H}$
if and only if $r>0.$
\end{proof}

The construction of the element $\widetilde{y}$ at the end of the previous
proof may be generalized, and it implies the following interesting corollary.

\begin{corollary}
\label{cor:dense}Assume that operators $\mathrm{P}^{f}$ and $\mathrm{R}$
commute. The set
\[
A=\left\{  y\in\mathcal{H}:\quad\int_{\mathcal{H}}d\left(  y\left\vert
x\right.  \right)  d\mu^{f}\left(  x\right)  =0\right\}  ,
\]
where $\mu^{f}\sim\mathcal{N}\left(  m^{f},\mathrm{P}^{f}\right)  ,$ and the
data likelihood is defined by (\ref{eq:likelihood}), is dense in $\mathcal{H}$
if the the operator $\mathrm{R}$ does not have positive lower bound.
\end{corollary}

\begin{proof}
To show that $A$ is dense it is sufficient to show that for each
$z\in\mathcal{H}$ and any $\delta>0$
\[
A\cap\left\{  u\in\mathcal{H}:\left\vert z-u\right\vert <\delta\right\}
\neq\emptyset.
\]
Let $z\in\mathcal{H}$ and $\delta>0$. Similarly as in the previous proof,
denote by $\left\{  e_{i}\right\}  $ the total orthonormal set such that
\[
\mathrm{P}^{f}e_{i}=p_{i}e_{i}\ \text{and}\ \mathrm{R}e_{i}=r_{i}e_{i}.
\]
Because
\[
r=\inf\left\{  r_{i}\right\}  =0,
\]
there exists a subsequence $\left\{  r_{i_{k}}\right\}  _{k=1}^{\infty}$ such
that
\[
r_{i_{k}}\leq\frac{\delta^{2}}{2^{k}}%
\]
for all $k\in\mathbb{N}.$ Now, define $\widetilde{z}=\sum_{i=1}^{\infty
}\widetilde{z}_{i}e_{i}$ such that
\[
\widetilde{z}_{i}=%
\begin{cases}
\left\langle z,e_{i}\right\rangle +r_{i}^{\nicefrac{1}{2}} & \text{if }%
i\in\left\{  i_{k}\right\}  _{k\in\mathbb{N}},\\
\left\langle z,e_{i}\right\rangle  & \text{if }i\notin\left\{  i_{k}\right\}
_{k\in\mathbb{N},}%
\end{cases}
\]
so
\[
\left\vert z-\widetilde{z}\right\vert =\left(  \sum_{i=1}^{\infty}\left\vert
\left\langle z-\widetilde{z},e_{i}\right\rangle \right\vert ^{2}\right)
^{\nicefrac{1}{2}}=\left(  \sum_{k=1}^{\infty}r_{i_{k}}\right)
^{\nicefrac{1}{2}}\leq\delta.
\]
Using the same method as in the proof of Theorem \ref{thm:commute},
\[
\int_{\mathcal{H}}\exp\left(  -\frac{1}{2}\left\vert \widetilde{z}%
-x\right\vert _{\left(  \mathrm{R}^{\left(  t\right)  }\right)  ^{-1}}%
^{2}\right)  d\mu^{f}\left(  x\right)  >0
\]
if and only if
\begin{equation}
\left(  \sum_{i=1}^{\infty}\ln\left(  1+\frac{p_{i}^{f}}{r_{i}}\right)
\right)  +\left(  \sum_{i=1}^{\infty}\frac{\widetilde{z}_{i}^{2}}{r_{i}%
+p_{i}^{f}}\right)  <\infty. \label{eq:At:proof:sum:}%
\end{equation}
However, when
\begin{equation}
\lim_{i\rightarrow\infty}\frac{p_{i}^{f}}{r_{i}}=0, \label{eq:pi/ri}%
\end{equation}
then
\[
\sum_{i=1}^{\infty}\frac{\widetilde{z}_{i}^{2}}{r_{i}+p_{i}^{f}}\geq\sum
_{k=1}^{\infty}\frac{\widetilde{z}_{i_{k}}^{2}}{r_{i_{k}}+p_{i_{k}}^{f}}%
\geq\sum_{k=1}^{\infty}\frac{r_{i_{k}}}{r_{i_{k}}\left(  1+\frac{p_{i_{k}}%
^{f}}{r_{i_{k}}}\right)  }=\infty.
\]
Therefore, using the same arguments as in the previous proof, if
(\ref{eq:pi/ri}) is not satisfied, then%
\[
\sum_{i=1}^{\infty}\ln\left(  1+\frac{p_{i}^{f}}{r_{i}}\right)  =\infty.
\]
Therefore, the sum at the left-hand side of (\ref{eq:At:proof:sum:}) diverges,
and $\widetilde{z}\in A.$
\end{proof}

\begin{remark}
\label{rem:agapiou} \cite[Theorem~3.8]{Agapiou-2015-ISC} have
shown that if the spectrum of $\mathrm{A}=\mathrm{R}^{-\nicefrac{1}{2}}
\mathrm{PR}^{-\nicefrac{1}{2}}$ consists of countably many eigenvalues (plus
zero), then the analysis measure $\mu^{a}$ is well defined and absolutely
continuous with respect to the forecast measure $\mu^{f}$ for $\eta$-almost
all $y$ if and only if $\mathrm{A}$ is trace class. The data error
distribution $\eta= \mathcal{N}\left(  0,\mathrm{R}\right)  $ is understood as
a Gaussian measure on a Hilbert space $\mathcal{X}\supset\mathcal{H}$. The
space $\mathcal{X}$ has a weaker topology than $\mathcal{H}$, and draws from
$\eta$ may not be in $\mathcal{H}$.

For example, suppose that $\mathrm{P}$ is trace class and $\mathrm{R}%
=\mathrm{I}$. Then $\mathrm{A}=\mathrm{P}$ is trace class, $\eta
=\mathcal{N}\left(  0,\mathrm{R}\right)  $ is white noise with the
Cameron-Martin space $\mathcal{H}$, and since the measure $\eta$ of the
Cameron-Martin space of $\eta$ is zero, data vector $y$ drawn from
$\mathcal{N}\left(  0,\mathrm{R}\right)  $ on $\mathcal{X}$ is in fact $\eta
$-a.s. not in $\mathcal{H}$. Consequently, in this example, \cite[Theorem~3.8]%
{Agapiou-2015-ISC} is not informative about the well-posedness of the analysis
measure $\mu^{a}$ when $y\in\mathcal{H}$, where the problem is formulated. In
the present approach, the data error distribution $\mathcal{N}\left(
0,\mathrm{R}\right)  $ is only a cylindrical measure on $\mathcal{H}$, and the
analysis measure $\mu^{a}$ is well defined and absolutely continuous with
respect to the forecast measure $\mu^{f}$, for all $y\in\mathcal{H}$.
\end{remark}

\bibliographystyle{amsplain}
\bibliography{../../references/extra,../../references/geo,../../references/other}

\end{document}